\documentclass[]{amsart}

\usepackage[bookmarksnumbered,colorlinks, plainpages]{hyperref}
\hypersetup{colorlinks=true,linkcolor=red, anchorcolor=green, citecolor=cyan, urlcolor=red, filecolor=magenta, pdftoolbar=true}

\usepackage{amsmath, amsthm, amscd, amsfonts, amssymb, graphicx, color, enumerate, xcolor, mathrsfs, latexsym}

\theoremstyle{plain}
\newtheorem{theorem}{Theorem}[section]
\newtheorem{lemma}[theorem]{Lemma}
\newtheorem{corollary}[theorem]{Corollary}
\newtheorem{conjecture}{Conjecture}

\theoremstyle{definition}
\newtheorem{definition}{Definition}[section]

\newtheorem{question}{Question}

\theoremstyle{remark}

\begin{document}
\title{Fully reducible simple Venn diagrams}
\author{M. Farrokhi D. G.}
\date{}
\address{Faculty of Mathematical Sciences, Ferdowsi University of Mashhad, Iran}
\email{m.farrokhi.d.g@gmail.com}
\begin{abstract}
We generalize Venn diagrams in spaces of arbitrary dimension $\geq 2$ and study simple Venn diagrams with the reducing property. Three equivalent conditions for a simple Venn diagram to reduce it completely and a classification of those diagrams is discussed. For example, we show that a simple $m$-dimensional $n$-Venn diagram is fully reducible if $n\leq m+1$ and conjecture that the converse is also true. An application of the generalized Venn diagrams and some more open problems are given.
\end{abstract}
\maketitle
%===================================================
\section{Introduction}
The notion of Venn diagrams first considered by John Venn in 1880 and a formal definition of a Venn diagram first introduced by Branko Gr\"{u}nbaum \cite{bg}, as a set of $n$ closed simple curves in the plane which make the plane into $2^n$ non-empty connected regions that are related to the intersection of interior and exterior of the curves. In particular, there are only finitely many points of intersections between curves. Such a set of curves is called an $n$-Venn diagram or simply a Venn diagram. Also a simple Venn diagram is one in which there are no points in common with three curves. A similar definition can be recognized for a Venn diagram in spaces of higher dimensions as is introduced in section 2. In this paper, we are interested in the structure of simple Venn diagrams, which can be reduced completely, i.e., every subset is again a Venn diagram. For a survey of Venn diagrams and related topics see \cite{fr-mw}.
%===============================================================================================
\section{Venn diagrams in Euclidian spaces}
As the aim of this note, we consider the notion of Venn diagrams in spaces of dimension $\geq 2$. For convenience, we set some conventions throughout this paper:

(1) An \textit{$m$-space} is a bounded or unbounded subset of a Euclidian space (the underlying space), which can be transformed continuously to an $m$-dimensional rectangle or a subspace, respectively.

(2) An \textit{$m$-surface} is a subset of an $m$-space as underlying space, which can be transformed to an $(m-1)$-dimensional sphere. An $m$-surface is assumed to be empty set when $m\leq0$.

(3) For an $m$-surface $S$, by $S^0$ and $S^1$ we mean the interior and exterior of $S$, respectively.
%--------------------------------------------------
\begin{definition}\label{d-1}
Let $S_1,\ldots,S_n$ be $m$-surfaces and $S$ be an $m$-space. Suppose that $S_i$'s make $S$ into $2^n$ non-empty connected subsets $S_1^{\varepsilon_1}\cap\cdots\cap S_n^{\varepsilon_n}$, which are called \textit{regions} and can be transformed continuously to the interior or exterior of an $m$-surface, where $\varepsilon_i=0,1$. In particular, if the intersection of any two surfaces is a union of finitely many $(m-1)$-surfaces, then the set $\{S_1,\ldots,S_n\}$ is called an \textit{$m$-dimensional Venn diagram} or simply a \textit{Venn diagram}.
\end{definition}
%--------------------------------------------------
\begin{definition}\label{d-2}
If $V=\{S_1,\ldots,S_n\}$ is an $m$-dimensional Venn diagram such that the intersection of any $k$ surfaces is a union of finitely many $(m-k+1)$-surfaces, then $V$ is called a \textit{simple Venn diagram}.
\end{definition}

At the first step, we need to sure about the existence of generalized Venn diagrams for any number of surfaces and finite dimensions greater than $1$. 
%--------------------------------------------------
\begin{theorem}\label{t-1}
There exists a simple $m$-dimensional $n$-Venn diagram for each $m\geq 2$ and $n\geq 1$.
\end{theorem}
\begin{proof}
The case $m=2$ is proved by John Venn \cite{jv} and Anthony Edwards \cite{awfe1,awfe2} using an inductively construction. Now let $V=\{S_1,\ldots,S_n\}$ be a simple $m$-dimensional Venn diagram. Then $V'=\{S_1',\ldots,S_n'\}$ is a simple $(m+1)$-dimensional Venn diagram, where $S_i'$ is the boundary of $\bar{S}_i\times [1-\frac{1}{i},2-\frac{1}{i}]$ with $\bar{S}_i=S_i\cup S_i^0$ for each $1\leq i\leq n$.
\end{proof}
%--------------------------------------------------
\begin{definition}\label{d-3}
Let $V=\{S_1,\ldots,S_n\}$ be an $m$-dimensional Venn diagram. Then $V$ is called \textit{reducible} if $V$ has an $(n-1)$-subset that is a Venn diagram. Also, $V$ is called \textit{fully reducible} if $W$ is a Venn diagram for every subset $W$ of $V$.
\end{definition}

As an example of a simple Venn diagram that is not reducible see \cite{fr-mw}.

Our ultimate goal is to give some equivalent conditions for a Venn diagram to be fully reducible, see Theorems \ref{t-2}, \ref{t-3} and \ref{t-4}. Theorem \ref{t-2} distinguishes the extra conditions in definition of fully reducible simple Venn diagrams. Theorem \ref{t-3} states that a simple Venn diagram is fully reducible if it possesses the maximum possible number of edges and Theorem \ref{t-4}, under the assumption that its converse is true, gives the complete classification of fully reducible simple Venn diagrams, which depends only on the dimension of the underlying space. Note that when surfaces intersect in a Venn diagram, they broke into some connected parts, which we called them \textit{edges} of the Venn diagram. In the following, by a (simple) Venn diagram we mean a Venn diagram of arbitrary dimension $\geq 2$.

To prove the main theorems, we need some lemmas, which we consider in the following.
%----------------------------------
\begin{lemma}\label{l-1}
Let $V=\{S_1,\ldots,S_n\}$ be a simple Venn diagram and $W$ be a subset of $V$ of size $k$. Then $W$ contains all $2^k$ possible regions.
\end{lemma}
\begin{proof}
Clearly, if $W=\{S_{i_1},\ldots,S_{i_k}\}$ and $S_{i_1}^{\varepsilon_1}\cap\cdots\cap S_{i_k}^{\varepsilon_k}=\varnothing$ then
\[S_1^{\varepsilon_1}\cap\cdots\cap S_n^{\varepsilon_n}=S_{i_1}^{\varepsilon_1}\cap\cdots\cap
S_{i_k}^{\varepsilon_k}\cap\bigcap_{i\in\{1,\ldots,n\}\setminus\{i_1,\ldots,i_m\}}S_i^{\varepsilon_i}=\varnothing,\]
where $\varepsilon_i\in\{0,1\}$ is arbitrary for each $i\notin\{i_1,\ldots,i_k\}$, a contradiction.
\end{proof}
%--------------------------------------------------
\begin{lemma}\label{l-2}
Let $V=\{S_1,\ldots,S_n\}$ be a fully reducible simple $m$-dimensional Venn diagram. Then $W=\{S_1\cap S_n,\ldots,S_{n-1}\cap S_{n}\}$ is an $(m-1)$-dimensional Venn diagram embedded on the $m$-surface $S_n$.
\end{lemma}
\begin{proof}
Since $V$ is fully reducible, $\{S_i,S_n\}$ is a simple Venn diagram for each $1\leq i<n$. Thus $S_i\cap S_n$ is connected and it can be transformed continuously to an $(m-1)$-surface. On the other hand, $\{S_1,\ldots,S_{n-1}\}$ is a simple Venn diagram and $S_n$ divides each region in this diagram into two new regions so that it possesses a unique edge in each region. This implies $W=\{S_1\cap S_n,\ldots,S_{n-1}\cap S_n\}$ contains exactly $2^{n-1}$ distinct regions. Moreover, $(S_{i_1}\cap S_n)\cap\cdots\cap (S_{i_k}\cap S_n)=S_{i_1}\cap\cdots\cap S_{i_k}\cap S_n$ is a union of finitely many $m-(k+1)+1=((m-1)-k+1)$-surfaces. Hence $W$ is a simple $(m-1)$-dimensional Venn diagram. Since every subset of $W$ is itself a Venn diagram as $V$ is fully reducible, it follows that $W$ is a fully reducible simple Venn diagram.
\end{proof}
%--------------------------------------------------
\begin{lemma}\label{l-3}
Let $V=\{S_1,\ldots,S_n\}$ be a simple Venn diagram and $A$ be a region. Then every edges of $A$ belongs to distinct surfaces.
\end{lemma}
\begin{proof}
Assume that $A$ is a region with two distinct edges belonging to $S_i$ for some $i$. Omitting $S_i$ from $V$, the region $A$ extends to a region $B$ in $V\setminus\{S_i\}$. Now $S_i$ divides $B$ into at least three parts and we should have an unconnected region when $S_i$ is joined again to $V\setminus\{S_i\}$, which is impossible.
\end{proof}

Note that, Lemma \ref{l-3} is in fact the restatement of \cite[Lemma 4.6]{kbc-ph-rep} to arbitrary dimensions.
%===================================================
\section{Main theorems}
Now we are able to prove our main results. As mentioned before, Theorem \ref{t-2} can be applied to replace conditions used to define fully reducible simple Venn diagrams by a subtancially smaller sets of conditions.
%--------------------------------------------------
\begin{theorem}\label{t-2}
Let $V=\{S_1,\ldots,S_n\}$ be a simple Venn diagram and $1<r<n$. Then $V$ is fully reducible if and only if every subset of $V$ of size $r$ is a Venn diagram.
\end{theorem}
\begin{proof}
Clearly, if $V$ is fully reducible, then every subset of $V$ of size $r$ is a Venn diagram. Now suppose that $1<r<n$ is a fixed number and that every subset of $V$ of size $r$ is a Venn diagram. To prove $V$ is fully reducible we first consider the following cases:

(1) Suppose $1<r<n$ and the subsets of $V$ of size $r$ are Venn diagrams. Let $W=\{S_{i_1},\ldots,S_{i_{r-1}}\}$ be a subset of $V$ and $S_{i_r},S_{i_{r+1}}\in V\setminus W$ be two different surfaces. Also let $W'=W\cup\{S_{i_r},S_{i_{r+1}}\}$. Since $W'\setminus\{S_{i_{r+1}}\}$ is a Venn diagram, $S_{i_{r+1}}$ divides all regions of $W'\setminus\{S_{i_{r+1}}\}$ and hence that of $W=W'\setminus\{S_{i_r},S_{i_{r+1}}\}$. Now $W\cup\{S_{i_{r+1}}\}$ is a Venn diagram with all regions connected. Since $W$ has all $2^{r-1}$ possible regions by Lemma \ref{l-1} it follows that $W$ is a Venn diagram. Hence, every subset of $V$ of size $r-1$ is a Venn diagram.

(2) If $r=2$, then by invoking Lemmas \ref{l-1} and \ref{l-3} and checking all possible cases we can see that all three surfaces form a Venn diagram.

(3) Assume that $m=2$, $1<r<n$, and every subset of $V$ of size $r$ is a Venn diagram. Let $W=\{S_{i_1}\ldots,S_{i_{r+1}}\}$ be a subset of $V$. Then $W'=W\setminus\{S_{i_{r+1}}\}$ is a Venn diagram, which implies that $S_{i_{r+1}}$ divides all regions of $W'$. To prove that $W$ is a Venn diagram it is enough to show that $S_{i_{r+1}}$ divides every region of $W'$ into exactly two parts. Let $A$ be a region of $W'$ and suppose that $S_{i_{r+1}}$ divides $A$ into at least three parts. Then $S_{i_{r+1}}$ has a least two edges in $A$. If $A$ possesses an edge of $S_{i_j}$, which is not divided by all edges of $S_{i_{r+1}}$ in $A$ as it is shown in Figure \ref{fig}, then by omitting $S_{i_j}$ from $W'$, $A$ extends to a region $B$, which is divided by $S_{i_{r+1}}$ more than twice. However, this is impossible by Lemma \ref{l-3} and the fact that $(W'\setminus\{S_{i_j}\})\cup\{S_{i_{r+1}}\}$ is a Venn diagram. Hence we can assume that every edge of $A$ is divided by all edges of $S_{i_{r+1}}$ in $A$. From this, we get that $A$ has only two edges, since every edges of $S_{i_{r+1}}$ enters from one edge and exits from another edge of $A$. If $r>2$, then there exists a surface $S_{i_j}$ of $W$, which has no edges in common with $A$ and omitting $S_{i_j}$ from and joining $S_{i_{r+1}}$ to $W'$, we get again a similar contradiction to Lemma \ref{l-3}. Thus $r=2$ contradicting part (2). Therefore, $S_{i_{r+1}}$ divides every region of $W'$ into exactly two new regions that is $W$ is a Venn diagram.

\begin{figure}[htbp]
\begin{picture}(250,150) \thinlines
\qbezier(0,50)(60,70)(75,150) \qbezier(0,80)(50,50)(60,0)
\qbezier(40,0)(100,70)(250,90) \qbezier(25,150)(60,70)(250,125)
\qbezier(160,150)(100,60)(210,0) \qbezier(210,150)(140,80)(250,65)
\qbezier(26,61)(53,80)(65.5,111.5) \qbezier(26,63)(50,40)(57,17)
\qbezier(55,17)(100,60)(195,82)
\qbezier(64,111)(105,93)(183,107.5)
\qbezier(182,108.5)(178,92)(195,80.5)
\qbezier(138.5,103)(135,83)(142,65)
\qbezier(141.5,103)(138,83)(144.5,66) \thicklines
\qbezier(0,110)(90,80)(80,0) \qbezier(150,0)(70,120)(250,80)
\put(25,102){$S_{i_{k+1}}$} \put(114,20){$S_{i_{k+1}}$}
\put(162,40){$S_{i_j}$} \put(80,70){$A$} \put(120,80){$B$}
\end{picture}
\caption{}\label{fig}
\end{figure}

(4) Suppose that in an $m$-dimensional space if every set of $k$ surfaces ($k>1$) form a Venn diagram and every set of $k+1$ surfaces contain all $2^{k+1}$ possible regions, then every set of $k+1$ surfaces form a Venn diagram. Let $S_{i_1},\ldots,S_{i_{r+1}}$ be surfaces in an $(m+1)$-dimensional space, which contain all $2^{r+1}$ possible regions and assume that all $r$ surfaces form a Venn diagram. Then Lemma \ref{l-2} shows that $S_{i_1}\cap S_{i_{r+1}},\ldots,S_{i_{r}}\cap S_{i_{r+1}}$ satisfy the induction hypothesis for $k=r-1$. Thus $S_{i_1}\cap S_{i_{r+1}},\ldots,S_{i_r}\cap S_{i_{r+1}}$ form a Venn diagram on $S_{i_{r+1}}$. Now $\{S_{i_1},\ldots,S_{i_{r+1}}\}$ is a Venn diagram, since $\{S_{i_1},\ldots,S_{i_r}\}$ is a Venn diagram with $2^r$ regions and moreover $S_{i_{r+1}}$ has exactly $2^r$ edges in the regions belonging to Venn diagram $\{S_{i_1},\ldots,S_{i_r}\}$, i.e., $S_{i_{r+1}}$ divides every region of $\{S_{i_1},\ldots,S_{i_r}\}$ into exactly two new regions.

Now using (1) every $r-1,r-2,\ldots,2,1$ surfaces form a Venn diagram and using (2), (3) and (4) together, every $r+1,r+2,\ldots,n$ surfaces form a Venn diagram. Therefore, $V$ is a fully reducible simple Venn diagram.
\end{proof}
%--------------------------------------------------
\begin{corollary}\label{c-1}
If $V$ is not a fully reducible simple $n$-Venn diagram, then for each $1<r<n$, there is a subset of $V$ of $r$ surfaces that is not a Venn diagram. In this case, Lemma \ref{l-1} assures the existence of more than $2^r$ regions, i.e., at least there is one unconnected region.
\end{corollary}
%--------------------------------------------------
Another consequence of the preceding theorem is the next result, which holds in any dimension. In what follows, we set the following notations for a diagram $V$:
\begin{itemize}
\item $r(V)$ is the number of regions of $V$;
\item $e(V)$ is the number of edges of $V$; and
\item $e_V(S)$ is the number of edges of $S$ in $V$, for each $S\in V$.
\end{itemize}
%--------------------------------------------------
\begin{theorem}\label{t-3}
Let $V=\{S_1,\ldots,S_n\}$ be a simple Venn diagram and $n\geq 2$. Then $e(V)\leq n2^{n-1}$ and the equality holds if and only if $V$ is fully reducible.
\end{theorem}
\begin{proof}
If $1\leq i\leq n$, then $S_i$ divides every region of $V\setminus\{S_i\}$ into at most two parts so that $S_i$ has at most one edge in every region of $V\setminus\{S_i\}$. The bijection between the regions divided by $S_i$ and the edges of $S_i$ gives
\[r(V)-r(V\setminus\{S_i\})=e_V(S_i).\]
Since $S_i$ may not divide some regions of $V\setminus\{S_i\}$, we have in general $r(V)\leq 2r(V\setminus\{S_i\})$, which gives
\[e_V(S_i)=r(V)-r(V\setminus\{S_i\})\leq r(V)-\frac{r(V)}{2}=2^n-2^{n-1}=2^{n-1}.\]
Hence
\[e(V)=e_V(S_1)+\cdots+e_V(S_n)\leq n2^{n-1}.\]
For the second part, first assume that $V$ is a fully reducible Venn diagram, then $V\setminus\{S_i\}$ is a Venn diagram for each $1\leq i\leq n$. Since $r(V\setminus\{S_i\})=2^{n-1}$, we have $e_V(S_i)=2^{n-1}$ and so $e(V)=n2^{n-1}$. Conversely, if $e(V)=n2^{n-1}$, then we should have $e_V(S_i)=2^{n-1}$ or $r(V\setminus\{S_i\})=2^{n-1}$, for each $1\leq i\leq n$. This in conjunction with Lemma \ref{l-1} implies that $V\setminus\{S_i\}$ is a Venn diagram. Now, by using Theorem \ref{t-2}, $V$ is a fully reducible simple Venn diagram and the proof is complete.
\end{proof}

The following theorem characterizes all fully reducible simple Venn diagrams in presence of its converse and shows that simple Venn diagrams have a more simple structure when they are lifted up to a space of higher dimension from a given one by a method like the one used in Theorem \ref{t-1}.
%--------------------------------------------------
\begin{theorem}\label{t-4}
Let $V=\{S_1,\ldots,S_n\}$ be an $m$-dimensional simple Venn diagram. If $V$ is fully reducible, then $n\leq m+1$.
\end{theorem}
\begin{proof}
Suppose that $V$ is fully reducible. Clearly, $n\leq m+1$ if $n=1$. Thus we may assume that $n>1$. If $m=2$, then it is known that $e(V)=2^{n+1}-4$ (see \cite{fr-mw}). Hence by Theorem \ref{t-3}, $V$ is fully reducible if and only if $2^{n+1}-4=n2^{n-1}$, which holds only for $n=2,3$.  Now suppose that $m\geq 3$ and the result holds for all Venn diagrams of dimension $m-1$. Since $V$ is fully reducible, Lemma \ref{l-2} follows that $\{S_1\cap S_n,\ldots,S_{n-1}\cap S_n\}$ is an $(m-1)$-dimensional fully reducible simple Venn diagram. Thus $n-1\leq (m-1)+1$ or $n\leq m+1$, as required.
\end{proof}

As we mentioned before, we strongly believe that the converse of Theorem \ref{t-4} is true. 
%--------------------------------------------------
\begin{conjecture}\label{o-1}
If $V$ is a simple $m$-dimensional $n$-Venn diagram with $n\leq m+1$, then $V$ is fully reducible.
\end{conjecture}
%===================================================
\section{More open problems}
We generalized the notion of Venn diagrams to higher dimensional spaces and considered a special problem. Thus we should leave many questions open. Here we state some of them, which seems to be of more importance.
%--------------------------------------------------
\begin{conjecture}\label{o-2}
Every two simple $m$-dimensional $n$-Venn diagrams have the same number of edges for all $m\geq 2$ and $n\geq 1$.
\end{conjecture}
%--------------------------------------------------
\begin{conjecture}\label{o-3}
If $V$ is a simple $m$-dimensional $n$-Venn diagram, then
\[e(V)\leq m2^n+a_0+a_1n+\cdots+a_{m-2}n^{m-2},\]
where the coefficients $a_0,a_1,\ldots a_{m-2}$ satisfy the equation
\[
\left[\begin{array}{ccccc}
1&2&2^2&\ldots&2^{m-2}\\
1&3&3^2&\ldots&3^{m-2}\\
\vdots&\vdots&\vdots&\ddots&\vdots\\
1&m&m^2&\ldots&m^{m-2}
\end{array}\right]
\left[\begin{array}{c}
a_0\\
a_1\\
\vdots\\
a_{m-2}
\end{array}\right]
= \left[\begin{array}{c}
2\cdot2^1-m\cdot2^2\\
3\cdot2^2-m\cdot2^3\\
\vdots\\
m\cdot2^{m-1}-m\cdot2^m
\end{array}\right].
\]
\end{conjecture}

It is not difficult to construct a simple $3$-dimensional $5$-Venn diagram with $73$ edges, while the upper bound in the above is $76$. Hence in the forementioned inequality, $e(V)$ does not attain the upper bound in general.

However, in presence of Conjecture \ref{o-1}, we can show that the equality occurs in Conjecture \ref{o-3} for simple $m$-dimensional $n$-Venn diagrams when $n\leq m+1$. To see this, let $V_{m,n}=\{S_1,\ldots,S_n\}$ be a simple $m$-dimensional $n$-Venn diagram. If $V_{i,j}=\{S'_1,\ldots,S'_j\}$ is given, we define $V_{i,j-1}$ and $V_{i-1,j-1}$ by
\[V_{i,j-1}=\{S'_1,\ldots,S'_{j-1}\}\]
and
\[V_{i-1,j-1}=\{S'_1\cap S'_j,\ldots, S'_{j-1}\cap S'_j\}.\]
Now, if $n\leq m+1$, then Conjecture \ref{o-1} and Lemma \ref{t-3} together imply that $V_{m-1,n-1}$ is a simple $(m-1)$-dimensional Venn diagram. Moreover, $V_{m,n-1}$ is also a Venn diagram. Iterating this way, we see that all $V_{i,j}$ are simple Venn diagrams, where $i=2,\ldots,m$, $j=1,\ldots,n$, and $j-i\leq n-m$.

We proceed by induction on $m$ to prove the equality in Conjecture \ref{o-3} when $n\leq m+1$. For the case $m=2$, we have the equality $e(V)=2^{n+1}-4=m2^n-a_0$ for every simple $2-$dimensional $n$-Venn diagram. Suppose $m\geq 3$ and there exist $a_0,a_1,\ldots,a_{m-3}$, such that $e(V)=(m-1)2^n+a_0+a_1n+\cdots+a_{m-3}n^{m-3}$ for all simple $(m-1)$-dimensional $n$-Venn diagram $V$, where $n\leq m$. Now, if $V=V_{m,n}$ is as above with $n\leq m+1$, then we get
\begin{eqnarray*}
e(V_{m,n})&=&e(V_{m,n-1})+\left[e(V_{m-1,n-1})+2^{n-1}\right]\\
&=&e(V_{m,n-2})+\left[e(V_{m-1,n-2})+2^{n-2}\right]+\left[e(V_{m-1,n-1})+2^{n-1}\right]\\
&\vdots&\\
&=&e(V_{m,2})+\left[e(V_{m-1,2})+2^2\right]+\cdots+\left[e(V_{m-1,n-1})+2^{n-1}\right]\\
&=&4+\sum_{i=2}^{n-1}\left[e(V_{m-1,i})+2^i\right]\\
&=&m2^n+4-4m+\sum_{i=2}^{n-1}\left[a_0+a_1i+\cdots+a_{m-3}i^{m-3}\right]\\
&=&m2^n+4-4m+a_0\sum_{i=2}^{n-1}1+a_1\sum_{i=2}^{n-1}i+\cdots+a_{m-3}\sum_{i=2}^{n-1}i^{m-3}\\
&=&m2^n+b_0+b_1n+\cdots+b_{m-2}n^{m-2},
\end{eqnarray*}
for some $b_0,b_1,\ldots,b_{m-2}$, which completes the proof. Note that, the coefficients can be computed from the equation $e(V)=n2^{n-1}$, for $n\leq m+1$.

Now, if $V$ is a simple $(m-1)$-dimensional $n$-Venn diagram with $n\leq m$, then, in conjunction with Conjecture \ref{o-1} and the statements above, we have
\[n2^{n-1}=e(V)=(m-1)2^n+a_0+a_1n+\cdots+a_{m-3}n^{m-3},\] 
for $n=1,\ldots,m$. Take $a_{m-2}=0$ and put $A$ , $X$, and $Y$ to be
\[\left[\begin{array}{cccc}
1&2&\cdots&2^{m-2}\\
1&3&\cdots&3^{m-2}\\
\vdots&\vdots&\ddots&\vdots\\
1&m&\cdots&m^{m-2}
\end{array}\right]
,\ 
\left[\begin{array}{c}
a_0\\
a_1\\
\vdots\\
a_{m-2}
\end{array}\right],\ 
\mbox{and}\left[\begin{array}{c}
2\cdot2^1-(m-1)\cdot2^2\\
3\cdot2^2-(m-1)\cdot2^3\\
\vdots\\
m\cdot2^{m-1}-(m-1)\cdot2^m
\end{array}\right],
\]
respectively. Then $AX=Y$ so that $X=A^{-1}Y$. Equalizing the last rows of $X$ and $A^{-1}Y$ together with the fact that $a_{m-2}=0$ and using determinant expansion rules on columns we obtain the following identity
\[\left|\begin{array}{ccccc}
1&2&\cdots&2^{m-3}&2\cdot2^1\\
1&3&\cdots&3^{m-3}&3\cdot2^2\\
\vdots&\vdots&\ddots&\vdots&\vdots\\
1&m&\cdots&m^{m-3}&m\cdot2^{m-1}\\
\end{array}\right|
=2(m-1)\left|\begin{array}{ccccc}
1&2&\cdots&2^{m-3}&2^1\\
1&3&\cdots&3^{m-3}&2^2\\
\vdots&\vdots&\ddots&\vdots&\vdots\\
1&m&\cdots&m^{m-3}&2^{m-1}\\
\end{array}\right|,
\]
for each $m\geq 3$.

As it is used before, a simple Jordan curve in the plane can divides only two edges of an $n$-gon, when it divides the $n$-gon into two parts. But the situation seems to be more complicated in $\mathbb{R}^3$. In fact, it seems that there is no limitation on the number of sides of a polyhedron in $\mathbb{R}^3$ in comparison to the edges in the plane.
%--------------------------------------------------
\begin{conjecture}\label{o-4}
For each polyhedron $P$ in $\mathbb{R}^3$ there exists a smooth simple closed curve $\mathcal{C}$ such that
\begin{itemize}
\item[(1)]$\mathcal{C}$ divides each face of $P$ into exactly two parts;
\item[(2)]$\mathcal{C}$ does not pass from the vertices of $P$;
\item[(3)]$\mathcal{C}$ is not tangent to edges of $P$.
\end{itemize}
\end{conjecture}
%--------------------------------------------------
\begin{question}\label{q-1}
What can be said about polyhedrons in $\mathbb{R}^n$ when $n>3$?
\end{question}
%--------------------------------------------------
\begin{question}\label{q-2}
What is the number of simple $m$-dimensional $n$-Venn diagrams?
\end{question}

We note that Edwards construction of simple Venn diagrams in the plane \cite{awfe1,awfe2}, shows that there are at least $n-2$ non-isomorphic simple Venn diagrams.

This paper is a part of my bachelor's thesis on Venn diagrams.
%===================================================

\end{document}